\documentclass[11pt,reqno]{amsart}

\usepackage{amsthm, amsmath, amssymb,booktabs,xcolor,graphicx}
\usepackage{bbm}
\usepackage{listings}
\usepackage{xcolor}
\usepackage{appendix}
\usepackage{subfig}

\definecolor{codegreen}{rgb}{0,0.6,0}
\definecolor{codegray}{rgb}{0.5,0.5,0.5}
\definecolor{codepurple}{rgb}{0.58,0,0.82}
\definecolor{backcolour}{rgb}{0.95,0.95,0.92}

\lstdefinestyle{mystyle}{
  backgroundcolor=\color{backcolour},   commentstyle=\color{codegreen},
  keywordstyle=\color{magenta},
  numberstyle=\tiny\color{codegray},
  stringstyle=\color{codepurple},
  basicstyle=\ttfamily\footnotesize,
  breakatwhitespace=false,         
  breaklines=true,                 
  captionpos=b,                    
  keepspaces=true,                 
  numbers=left,                    
  numbersep=5pt,                  
  showspaces=false,                
  showstringspaces=false,
  showtabs=false,                  
  tabsize=2
}

\usepackage[margin=1in]{geometry}

\usepackage[hidelinks]{hyperref}

\usepackage[utf8]{inputenc}
\usepackage[T1]{fontenc}

\usepackage[textsize=small,backgroundcolor=orange!20]{todonotes}

\usepackage[hidelinks]{hyperref}
\usepackage{url}

\usepackage[noabbrev,capitalize]{cleveref}
\crefname{equation}{}{}

\numberwithin{equation}{section}
\newtheorem{theorem}{Theorem}[section]
\newtheorem{proposition}[theorem]{Proposition}
\newtheorem{lemma}[theorem]{Lemma}
\newtheorem{claim}[theorem]{Claim}
\newtheorem{corollary}[theorem]{Corollary}
\newtheorem{conjecture}[theorem]{Conjecture}
\newtheorem*{question*}{Question}

\theoremstyle{definition}

\newtheorem{question}[theorem]{Question}
\newtheorem*{definition*}{Definition}

\theoremstyle{remark}
\newtheorem*{remark}{Remark}

\newcommand{\abs}[1]{\left\lvert#1\right\rvert}
\newcommand{\norm}[1]{\left\lVert#1\right\rVert}
\newcommand{\snorm}[1]{\lVert#1\rVert}
\newcommand{\ang}[1]{\left\langle #1 \right\rangle}
\newcommand{\sang}[1]{\langle #1 \rangle}

\newcommand{\ceil}[1]{\left\lceil #1 \right\rceil}
\newcommand{\paren}[1]{\left( #1 \right)}

\newcommand{\wh}{\widehat}

\newcommand{\CC}{\mathbb C}
\newcommand{\EE}{\mathbb E}

\newcommand{\PP}{\mathbb P}

\newcommand{\RR}{\mathbb R}
\newcommand{\ZZ}{\mathbb Z}

\newcommand{\mb}{\mathbb}
\newcommand{\mbf}{\mathbf}
\newcommand{\mbm}{\mathbbm}

\newcommand{\ol}{\overline}
\newcommand{\on}{\operatorname}

\DeclareMathOperator{\Tr}{Tr}
\newcommand{\HS}{\mathrm{HS}}
\DeclareMathOperator{\re}{Re}

\usepackage{etoolbox}
\usepackage{appendix}
\usepackage[capitalize]{cleveref}

\allowdisplaybreaks

\AtBeginEnvironment{appendices}{\crefalias{section}{appendix}}

\title{Cayley graphs without a bounded eigenbasis}
\author[Sah]{Ashwin Sah}
\author[Sawhney]{Mehtaab Sawhney}
\author[Zhao]{Yufei Zhao}

\thanks{Zhao was supported by NSF Award DMS-1764176, a Sloan Research Fellowship, and the MIT Solomon Buchsbaum Fund.}

\address{Department of Mathematics, Massachusetts Institute of Technology, Cambridge, MA 02139, USA}
\email{\{asah,msawhney,yufeiz\}@mit.edu}

\begin{document}

\begin{abstract}
Does every $n$-vertex Cayley graph have an orthonormal eigenbasis all of whose coordinates are $O(1/\sqrt{n})$?
While the answer is yes for abelian groups, we show that it is no in general.

On the other hand, we show that every $n$-vertex Cayley graph (and more generally, vertex-transitive graph) has an orthonormal basis whose coordinates are all $O(\sqrt{\log n / n})$, and that this bound is nearly best possible.

Our investigation is motivated by a question of Assaf Naor, who proved that random abelian Cayley graphs are small-set expanders, extending a classic result of Alon--Roichman.
His proof relies on the existence of a bounded eigenbasis for abelian Cayley graphs, which we now know cannot hold for general groups.
On the other hand, we navigate around this obstruction and extend Naor's result to nonabelian groups.
\end{abstract}

\maketitle

\section{Introduction}

\subsection{Bounded eigenbasis}

It is a fundamental problem to understand the spectral decomposition of a Cayley graph. Since every vertex in a Cayley graph has the same degree, it does not matter whether we are talking about the adjacency matrix or the Laplacian matrix, but we will stick with the adjacency matrix for concreteness.
Enormous attention has been given to the eigenvalues of Cayley graphs, especially the spectral gap, due to an intimate connection with the expansion properties of the graph \cite{HLW06,Lub12}.
In this paper, we study the complementary question of what can arise as eigenvectors of Cayley graphs.

We adopt the following normalization\footnote{This normalization, viewing $x$ as a function on a set or group equipped with the averaging measure, is different from the normalization used in the abstract, where we use the usual Euclidean distance in $\RR^n$.}. Given a finite set $S$ and a function $x \colon S \to \CC$ (sometimes viewed as a vector $x \in \CC^S$), we denote its $L^p$ norm by
\[
\norm{x}_{L^p(S)}:= \paren{\EE_{s \in S} \abs{x(s)}^p}^{1/p}.
\]
The Hermitian inner product is defined by $\sang{x, y} = \EE_{s \in S} \ol{x(s)} y(s)$. We say that $x$ is \emph{$C$-bounded} if $\snorm{x}_{L^\infty(S)} \le C$, and we say that a set of functions is $C$-bounded if all of its elements are $C$-bounded. In a unitary or orthonormal eigenbasis, each eigenfunction $x$ is normalized as $\snorm{x}_{L^2(S)} = 1$.

Below is the main question that we study. Here our Cayley graphs are unweighted and undirected. A Cayley graph on a finite group $G$ with symmetric generator $S = S^{-1}$ (not containing the identity) has edges of form $(g,sg)$ ranging over all $g\in G$ and $s \in S$.

\begin{question} \label{qn:main}
What is the minimum $C(n)$ so that every $n$-vertex Cayley graph has a $C(n)$-bounded unitary (or orthonormal) eigenbasis?
\end{question}

Every abelian Cayley graph has a $1$-bounded unitary eigenbasis. 
Indeed, given an abelian group $G$, the basis of Fourier characters of $G$ forms an eigenbasis for every Cayley graph on $G$. In particular, all coordinates of such a Fourier basis are are roots of unity.
The existence of a bounded eigenbasis is useful in certain applications.
In fact, the initial motivation for this work is a result of Naor~\cite{Nao12} that proves a certain small-set expansion property of random abelian Cayley graphs, extending a classic result of Alon--Roichman~\cite{AR94} that random Cayley graphs on arbitrary groups are expanders. 
Naor's argument uses that every abelian Cayley graph has a $1$-bounded eigenbasis.
He asks whether his results also hold for nonabelian groups.
Here we show that general Cayley graphs do not always have a bounded eigenbasis, therefore exhibiting an obstruction to Naor's argument for nonabelian groups.
On the other hand, we provide an alternative argument showing that Naor's theorem indeed extends to general groups. See \cref{thm:naor-nonabelian} below for a precise statement.

Our first result below implies that Cayley graphs do not always have a bounded unitary eigenbasis. 
It gives a lower bound $C(n) \gtrsim \sqrt{\log n}/\log\log n$ for \cref{qn:main} for infinitely many $n$. 
(Notation: we write $A \lesssim B$ and $A = O(B)$ to mean that $A \le C B$ for some constant $C > 0$.)

\begin{theorem} \label{thm:lower}
There exist infinitely many Cayley graphs $G$ whose adjacency matrix 
has an eigenspace all of whose eigenfunctions $x \colon G \to \CC$ satisfy $\norm{x}_{L^\infty(G)} \ge c\norm{x}_{L^2(G)} \sqrt{\log n} / \log\log n $, where $n$ is the number of vertices and $c > 0$ is some absolute constant.
\end{theorem}

The next result gives a nearly matching upper bound of $C(n) \lesssim \sqrt{\log n}$ for \cref{qn:main}.

\begin{theorem}\label{thm:upper-cayley}
Every Cayley graph has an orthonormal $C\sqrt{\log n}$-bounded eigenbasis, where $n$ is the number of vertices and $C$ is some absolute constant.
\end{theorem}

More generally, the same upper bound holds for vertex-transitive graphs.

\begin{theorem}\label{thm:upper-transitive}
Every vertex-transitive graph has an orthonormal $C\sqrt{\log n}$-bounded eigenbasis, where $n$ is the number of vertices and $C$ is some absolute constant.
\end{theorem}

It remains an intriguing open problem to close the gap between the upper and lower bounds.
This problem appears to be related to a recent deep and difficult result of Green~\cite{Gre20}, who showed that the maximum possible width of a finite transitive subset of the unit sphere in $\RR^d$ is on the order of $1/\sqrt{\log d}$ (in sharp constrast to infinite subsets, e.g., the entire sphere has width 1). Green's theorem answers a question of the third author, which was in turn prompted by \cite{CZ17} and this work.
Green's proof relies on the classification of finite simple groups.

Let us mention a few directions worth further investigation.
First, our construction proving \cref{thm:lower} uses graphs of increasing degree. Can one also find bounded degree Cayley graphs without a bounded eigenbasis?

\begin{conjecture}
There exists some $d$ such that for every $C$ there exists a $d$-regular Cayley graph without an orthonormal $C$-bounded eigenbasis.
\end{conjecture}

Another direction worth exploring further is to understand what families of groups always have bounded eigenbasis. 
Extending the example of abelian groups, it is not hard to show using nonabelian Fourier analysis that in a group where every irreducible representation has dimension at most $d$, every Cayley graph has a $\sqrt{d}$-bounded unitary eigenbasis. Given these examples, a natural question is if for more natural classes of ``nearly abelian'' groups, every Cayley graph has a bounded eigenbasis.
\begin{question}\label{ques:nilpotent-affine}
Do Cayley graphs on nilpotent groups of bounded step always have bounded eigenbasis? 
What about affine groups?
\end{question}

The general problem of characterizing groups with the bounded eigenbasis property is somewhat reminiscent of the characterization of approximate groups by Breuillard, Green, and Tao~\cite{BGT12}, which unifies classic theorems of Freiman on sets of bounded doubling \cite{Fre73} and Gromov on groups of polynomial growth \cite{Gro91}.

Let us mention that another instance where a bounded eigenbasis came in handy was in studying the relationship between discrepancy and eigenvalues of Cayley graphs. 
Kohayakawa, R\"odl, and Schacht~\cite{KRS16} showed that for abelian Cayley graphs, having small discrepancy is equivalent to having small second eigenvalues, with a spectral proof suggested by Gowers. 
The proof relies on the bounded eigenbasis of abelian Cayley graphs.
The abelian hypothesis was later removed by Conlon and Zhao~\cite{CZ17} via an application of Grothendieck's inequality.

The boundedness of eigenfunctions has an appealing interpretation for \emph{spectral graph drawings}. Hall's spectral drawing of a graph~\cite{Hal70} (also see Spielman's survey \cite{Spi12}, which contains some nice figures) places each vertex $v$ at $(x(v), y(v)) \in \RR^2$, where $x$ and $y$ are eigenfunctions corresponding to the second and third eigenvalues of the graph Laplacian (here $x$ and $y$ are assumed orthogonal and properly scaled). This drawing has the property that it minimizes the sum of squared edge-lengths among all drawings of the graph in $\RR^2$ with the vertices in isotropic position (so that $x$ and $y$ coordinates each have variance 1 and are uncorrelated). Every abelian Cayley graph has a spectral drawing where all the coordinates are bounded.
On the other hand, \cref{thm:lower} gives us an example of a Cayley graph where no spectral drawing can fit inside a disk of radius $c \sqrt{\log n}/\log\log n$ (provided that the eigenspace in the theorem corresponds to the second and third eigenvalues, which can be achieved; see the end of \cref{sec:construction} for further details). Some examples of spectral drawings of Cayley graphs used in the proof of \cref{thm:lower} are shown in \cref{fig:spectral-drawing}. 

\begin{figure}[t]
\centering
\subfloat[$G = S_3 \ltimes (\ZZ/2\ZZ)^3$]{{\includegraphics[width=7cm]{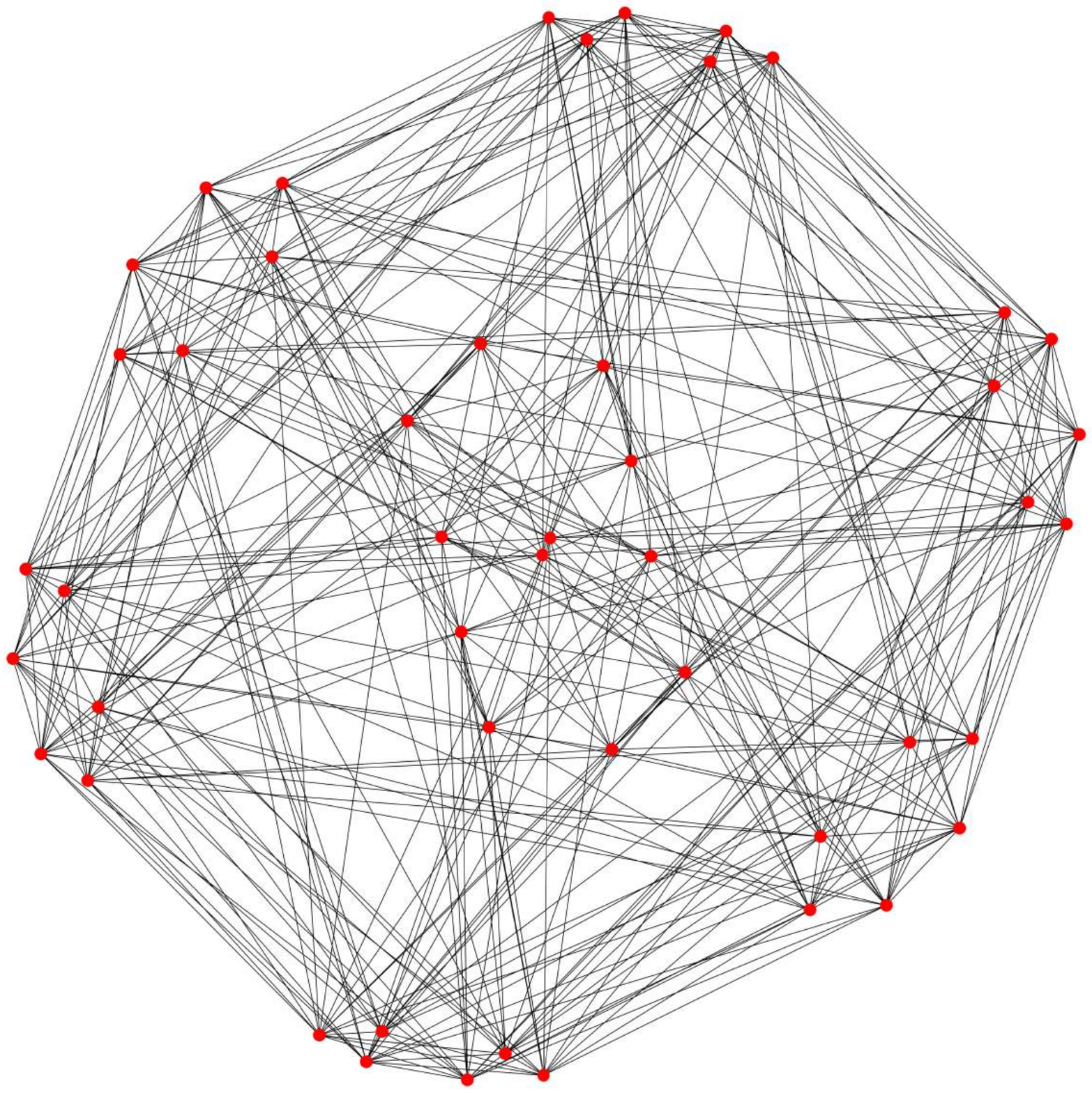}}}
\qquad
\subfloat[$G = S_4 \ltimes (\ZZ/2\ZZ)^4$]{{\includegraphics[width=7cm]{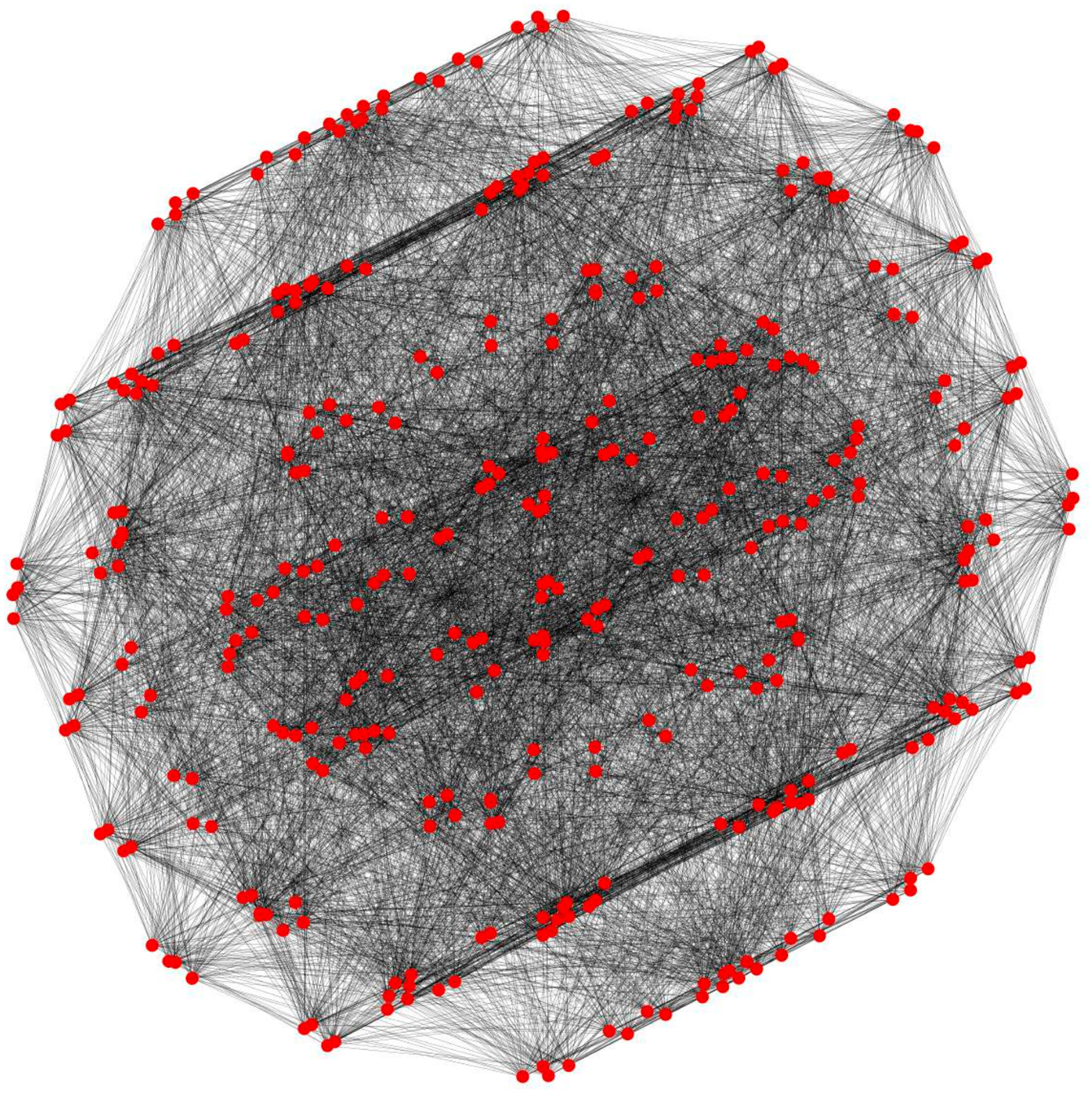}}}

\caption{Some spectral drawings of Cayley graphs used in the proof of \cref{thm:lower}, corresponding to the group $G = S_d \ltimes (\ZZ/2\ZZ)^d$, drawn for $d \in \{3,4\}$. Such a spectral drawing requires a canvas of side-lengths on the order of $\sqrt{d/\log d}$, where vertices have uncorrelated $x$ and $y$ coordinates each having variance 1. See the remark at the end of \cref{subsec:unweighted} on how these figures were generated.} \label{fig:spectral-drawing}
\end{figure}

\subsection{Random Cayley graphs are small-set expanders}

A classic result due to Alon and Roichman~\cite{AR94} shows that in a random Cayley graph of a group $G$ generated by $k > C\epsilon^{-2} \log \abs{G}$ independent and uniform random group elements, all eigenvalues other than the top one has absolute value at most $k\epsilon$. In particular, via the expander mixing lemma, it implies that for every $\emptyset \neq X \subsetneq G$,
\[
\abs{ \frac{e(X, G \setminus X) }{\frac{2k}{\abs{G}} \abs{X}\abs{G\setminus X}} - 1} \le \epsilon,
\]
where $e(A,B)$ counts the number of edges with one endpoint in $A$ and the other in $B$.

Naor~\cite{Nao12} developed a new Banach space-valued Azuma inequality and proved more refined small-set isoperimetry inequalities in random Cayley graphs of abelian groups, and he asked whether his result can be extended to all groups. Here we answer his question affirmatively. 
In the following theorem, by ``the Cayley (multi)graph associated $k$ independent uniformly chosen random group elements'' we mean the following: select random $g_1, \dots, g_k$ and take the Cayley graph generated by $g_1, g_1^{-1}, \dots, g_k, g_k^{-1}$, taken with multiplicity. Allowing multiplicities makes the result a bit easier to state, and is a technicality that one should feel free to ignore (in many parameter ranges multiplicities are unlikely to occur).

\begin{theorem} \label{thm:naor-nonabelian}
There exists a universal constant $C > 0$ such that for every positive integer $k$ and every group $G$, with probability at least $1/2$, the Cayley (multi)graph associated to $k$ independent uniformly chosen random group elements has the property that for every $X \subseteq G$ with $1 < \abs{X} \le \abs{G}/2$, the number of edges $e(X, G\setminus X)$ between $X$ and $G \setminus X$ satisfies
\[
\abs{ \frac{e(X, G \setminus X) }{\frac{2k}{\abs{G}} \abs{X}\abs{G\setminus X}} - 1} \le C \sqrt{\frac{\log \abs{X}}{k}}.
\]
\end{theorem}

Naor proved \cref{thm:naor-nonabelian} for abelian groups. His proof relies on a bounded eigenbasis of abelian Cayley graphs. In \cref{sec:small-set-expansion} we explain how to bypass this obstacle in order to prove the result for nonabelian groups.

\medskip

\noindent \textbf{Acknowledgments.} Zhao thanks Assaf Naor for discussions and for encouraging him to work on this problem. We thank Shengtong Zhang for pointing out some typographical errors.

\section{Preliminaries}\label{sec:background}

\subsection{Nonabelian Fourier transform}\label{subsec:fourier}
We begin by summarizing some standard facts on nonabelian Fourier analysis (e.g., \cite{GW10}). 
Given a finite group $G$, let $\wh{G}$ denote the set of irreducible unitary representations of $G$. 
For each representation $\rho \in \wh G$, call its dimension $d_\rho$, and call the space that it acts on $W_\rho \cong \CC^{d_\rho}$. 
For any $f: G\to\mb{C}$ and $\rho\in\wh{G}$, its Fourier transform evaluated at $\rho$ is given by
\[
\wh{f}(\rho) = \mb{E}_{g\in G}f(g)\rho(g),
\]
which is an endomorphism of $W_\rho$ (i.e., $\wh{f}(\rho)\in\on{End}W_\rho$).
There is an inversion formula, namely
\[f(g) = \sum_{\rho\in\wh{G}}d_\rho\langle\rho(g),\wh{f}(\rho)\rangle_\HS\]
where $\langle A,B\rangle_\HS = \on{Tr}(A^\dagger B)$ is the Hilbert–Schmidt inner product, 
which is just the entry-wise Hermitian product of matrices.
The Hilbert--Schmidt norm is written as $\|A\|_\HS = \sqrt{\on{Tr}(A^\dagger A)}$.
We have Parseval's identity
\[\langle f_1,f_2\rangle_{L^2(G)} = \mb{E}_{g\in G}\ol{f_1(g)}f_2(g) = \sum_{\rho\in\wh{G}}d_\rho\langle\wh{f_1}(\rho),\wh{f_2}(\rho)\rangle_\HS,\]
and in particular,
\[\mb{E}_{g\in G}|f(g)|^2 = \sum_{\rho\in\wh{G}}d_\rho\|\wh{f}(\rho)\|_\HS^2.\]
Finally, we define a convolution of two functions $f_1,f_2: G\to\mb{C}$ via
\[(f_1\ast f_2)(g) = \mb{E}_{h\in G}f_1(gh^{-1})f_2(h).\]
The Fourier transform turns convolution into matrix multiplication:
\[\wh{f_1\ast f_2}(\rho) = \wh{f_1}(\rho)\wh{f_2}(\rho)\]
for all $\rho\in\wh{G}$.

\subsection{Eigendecomposition}\label{subsec:eigendecomposition}
Given a function $f: G\to\mb{C}$, we consider the operator $M_f$ acting on $\CC^G$, the space of functions $G \to \CC$, via
\[
M_f x = f \ast x,
\]
i.e., $(M_f x)(g) = \EE_{h \in G} f(gh^{-1}) x(h)$ for all $x \colon G \to \CC$ and $g \in G$. 
Equivalently, one can also view $M_f$ as a matrix with rows and columns indexed by $G$, whose entry in position $(g,h) \in G \times G$ is $f(gh^{-1})/\abs{G}$.
Then, viewing $x \in \CC^G$ as a vector, the matrix product $M_f x$ agrees with the definition above. 
The matrix can be thought of as the adjacency matrix (after suitable normalization) of a Cayley graph.
Let us explain how to analyze the eigendata of $M_f$ using the Fourier transform.

Assume from now on that $f(g^{-1}) = \ol{f(g)}$ for every $g\in G$. Then
\[\wh{f}(\rho) = \mb{E}_{g\in G}f(g)\rho(g)\]
is Hermitian. 
For any $x: G\to\mb{C}$, applying the Fourier transform, we see that 
$x$ is an eigenfunction of $M_f$ with eigenvalue $\lambda$ (i.e., $f\ast x = \lambda x$) if and only if
\[\wh{f}(\rho)\wh{x}(\rho) = \lambda\wh{x}(\rho)\quad \text{ for all } \rho\in\wh{G},\]
i.e., all columns of $\wh{x}(\rho)$ (when viewed as a $d_\rho \times d_\rho$ matrix) lie in the eigenspace of $\wh{f}(\rho)$ corresponding to the eigenvalue $\lambda$.

Let $V_\rho$ be the subspace of functions whose Fourier transform is supported on $\rho$:
\begin{align}
V_\rho &= \{x\in L^2(G): \on{supp}\wh{x}\subseteq\{\rho\}\} \nonumber
\\
&= \{x \in L^2(G) : x(g) = d_\rho \ang{\rho(g), A} \text{ for some } A \in \on{End} W_\rho \}.
\label{eq:Vrho}
\end{align}
For any column vector $\mbf{v}\in W_\rho$, we define
\begin{align}
V_{\rho,\mbf{v}} 
&= \{ x \in V_\rho : \text{every column of } \wh x (\rho) \text{ is a multiple of }\mbf{v}\} \nonumber
\\
&= \{x\in L^2(G): x(g) = d_\rho\sang{\rho(g),\mbf{v}\mbf{w}^\dagger}_\HS\text{ for some  }\mbf{w}\in W_\rho\} \subseteq V_\rho.\label{eq:Vrhov}
\end{align}
In particular, if $\wh f(\rho) \mbf{v} = \lambda \mbf{v}$ for some $\lambda \in \RR$, then
$M_f x = \lambda x$ for all $x \in V_{\rho,\mbf{v}}$ (as can be seen by taking the Fourier transform).
Furthermore, if $\mbf{v}, \mbf{v}' \in W_\rho$ with $\mbf{v}^\dagger \mbf{v}' = 0$, then $\ang{x,x'} = 0$ for all $x \in V_{\rho, \mbf{v}}$ and $x' \in V_{\rho, \mbf{v}'}$.

To summarize, we have an orthogonal decomposition (the orthogonality is easy to check via the Fourier transform)
\[
    \CC^G = \bigoplus_\rho V_\rho.
\]
For each $\rho \in \wh G$, let $\mbf{v}_1^\rho,\ldots,\mbf{v}_{d_\rho}^\rho \in W_\rho$ be an eigenbasis of $\wh f(\rho) \in \on{End} W_\rho$, and call the corresponding eigenvalues $\lambda_{\rho,1},\ldots,\lambda_{\rho,d_\rho}$. We have an orthogonal decomposition
\[
    V_\rho = \bigoplus_{j=1}^{d_\rho}V_{\rho,\mbf{v}_j^\rho}
\]
and $M_f x = \lambda_j^\rho x$ for each $x \in V_{\rho, \mbf{v}_j^\rho}$.
Thus the eigenvalues of $M_f$ consists of $\lambda_{\rho, j}$ with multiplicity $d_\rho$, ranging over all $\rho \in \wh G$ and $j \in [d_\rho]$. The eigenspace of $M_f$ corresponding to an eigenvalue $\lambda$ is 
the direct sum of all $V_{\rho, \mbf{v}_j^\rho}$ ranging over all $\rho \in \wh G$ and $j \in [d_\rho]$ with $\lambda_{\rho, j} = \lambda$.

\subsection{Schatten norms} \label{sec:schatten}
The Schatten $p$-norm $\snorm{A}_{S_p}$ of a matrix $A\in\CC^{n\times n}$ is defined via
\[\snorm{A}_{S_p}^p = \sum_{i=1}^n\sigma_i(A)^p,\]
where $\sigma_1(A), \dots, \sigma_n(A)$ are the singular values of $A$.

The Schatten $p$-norm satisfies a noncommutative H\"older's inequality (e.g., \cite[Corollary~IV.2.6]{Bha97}): for $1\le p\le q\le \infty$ with $1/p + 1/q = 1$, we have 
\begin{equation}\label{eq:Hold}
\sang{A,B}_\HS\le\snorm{A}_{S_p}\snorm{B}_{S_q}.
\end{equation}

Given a function $f \colon G \to \CC$ on a finite group $G$,
we define its Schatten $p$-norm $\norm{f}_{S_p}$ to be the Schatten $p$-norm of its associated matrix $M_f$ (giving the linear map $x \mapsto f \ast x$ on $\CC^G$):
\begin{equation} \label{eq:Sp}
\norm{f}_{S_p} 
= \paren{\sum_i \sigma_i(M_f)^p}^{1/p}
= \paren{\sum_\rho d_\rho \snorm{\wh{f}(\rho)}_{S_p}^p}^{1/p}.
\end{equation}

\section{Construction} \label{sec:construction}

In this section, we prove \cref{thm:lower} by constructing a Cayley graph on a group $G$ with an eigenspace all of whose eigenfunctions satisfy $\norm{x}_{L^{\infty(G)}} \gtrsim \norm{x}_{L^2(G)} \sqrt{\log\abs{G}}/\log\log \abs{G}$. 

To motivate our construction, we first explain in \cref{subsec:unitary-group} what happens for the unitary group $G = U(d)$, which is simpler to analyze although it is not finite. 
Then, in \cref{subsec:weighted}, we explain how to construct an edge-weighted Cayley graph on a certain finite subgroup of $U(d)$.
We then explain in \cref{subsec:unweighted} how to convert the edge-weighted construction to an unweighted construction via sampling, and show that eigenvectors maintain their desired properties. 
Only \cref{subsec:unweighted} is required for the proof of \cref{thm:lower}, and the earlier subsections are solely for motivation, but we hope that they are helpful to the readers.

\subsection{Unitary group}\label{subsec:unitary-group}
Let $G = U(d)$.
Let $\rho$ denote the standard representation of $G$ on $\CC^d$, which is irreducible.
Let $V_\rho$ denote the subspace of $L^2(G)$ consisting of all $x \in L^2(G)$ of the form $x(g) = d \sang{\rho(g), A}_\HS$ for some $A \in \CC^{d\times d}$ , i.e., the Fourier transform $\wh{x}$ is supported at $\rho$ and $\wh{x}(\rho) = A$. Note that this definition of $V_\rho$ agrees with our earlier definition in \cref{eq:Vrho} for finite groups.

\begin{claim}\label{clm:init-inequality} For any $x \in V_\rho$, we have $\|x\|_{L^\infty(G)} \ge \sqrt{d} \|x\|_{L^2(G)}$.
\end{claim}
\begin{proof}
Let $A \in \CC^{d \times d}$ be such that $x(g) = d \sang{\rho(g), A}_\HS$. By Fourier inversion and Parseval, $\snorm{x}_{L^2(G)} = \sqrt{d}\snorm{A}_\HS$. Thus we have
\[
\|x\|_{L^\infty(G)} = d \sup_{U \in U(d)}\sang{U, A}_\HS = d \snorm{A}_{S_1} \ge d\snorm{A}_{S_2} = d \snorm{A}_\HS = \sqrt{d} \|x\|_{L^2(G)}
\]
by definition of $\rho$ and duality of the Schatten norms.
\end{proof}

For a sufficiently generic  $f \colon G \to \CC$ (assuming no unwanted eigenvalue collisions), 
the subspace $V_\rho$ is a direct sum of eigenspaces of the operator $M_f$, there by giving a continuous analog of \cref{thm:lower}. The actual construction proving \cref{thm:lower} will involve a discretization of this construction.

\subsection{Weighted construction}\label{subsec:weighted}
In this section we prove a weighted analogue of \cref{thm:lower} which serves as a stepping stone towards the entire proof. Recall from earlier that for a function $f: G\to\CC$ we have $M_f(g,h) = f(gh^{-1})/\abs{G}$.
\begin{theorem}\label{thm:main-signed-weighted}
There exist some constant $c > 0$ and infinitely many groups $G$ and functions $f:G\to \RR$ such that $M_f$ has an eigenspace all of whose elements $x$ satisfy 
\[
\norm{x}_{L^\infty(G)} \ge c\frac{\sqrt{\log \abs{G}} }{ \log\log \abs{G}} \norm{x}_{L^2(G)}.
\]
\end{theorem}

Let $G = S_d \ltimes (\ZZ/2\ZZ)^d$, where $S_d$ acts on $(\ZZ/2\ZZ)^d$ by permuting coordinates. 
The group has a natural $d$-dimensional representation on $\CC^d$, which we call $\rho$, where $S_d$ permutes the coordinates and $(\ZZ/2\ZZ)^d$ flips the signs of the coordinates. 
It is easy to check that $\rho$ is irreducible.

We need a $G$-orbit on the unit sphere in $\CC^d$ with large width in every direction. 
The next lemma serves as a finitary analogue of \cref{clm:init-inequality}.
This lemma also appears in \cite{Gre20} and we include its proof here for the convenience of the reader. 

\begin{lemma}\label{lem:sup-inequality}
Let $\mbf{a}$ be the unit vector in the direction of $(1,1/\sqrt{2},\ldots,1/\sqrt{d})$. Then for any $\mbf{v}\in\mb{C}^d$ we have 
\[\sup_{g\in G}|\langle\mbf{v},\rho(g)\mbf{a}\rangle|\gtrsim\frac{|\mbf{v}|}{\sqrt{\log d}}.\]
\end{lemma}

\begin{remark} 
A difficult recent result of Green~\cite{Gre20}, confirming a conjecture of Zhao, showed that for every \emph{finite} subgroup $G$ of $U(d)$ and every unit vector $\mbf{a} \in \CC^d$, there is some unit vector $\mbf{v}$ such that $\sup_{g \in G} \abs{\ang{\mbf{v}, \rho(g)\mbf{a}}} \lesssim 1/\sqrt{\log d}$ (i.e., a tight upper bound on the width of every finite transitive subset of a sphere).
In contrast, the width of an infinite transitive subset of the sphere can be as large as $1$ (e.g., the entire sphere).
It is initially quite counterintuitive that a finiteness assumption implies such a dramatic reduction in the width of an orbit.
\end{remark}

\begin{proof}
Let us first assume that $\mbf{v}\in\RR^d$, so
\[\sup_{g\in G}|\langle\mbf{v},\rho(g)\mbf{a}\rangle| = \sup_{g\in G}|\langle\rho(g)\mbf{v},\mbf{a}\rangle|\ge\langle\mbf{w},\mbf{a}\rangle,\]
where $\mbf{w}$ is the vector obtained by making the coordinates of $\mbf{v}$ nonnegative and then rearranging them in nonincreasing order. Let $\mbf{w} = (w_1,\ldots,w_d)$ with $w_1\ge\cdots\ge w_d\ge 0$. Then
\begin{align*}
\paren{\sum_{j=1}^{d}\frac{1}{j}}\langle\mbf{w},\mbf{a}\rangle^2 &= \left(\sum_{j=1}^d\frac{w_j}{\sqrt{j}}\right)^2\\
&\gtrsim w_1^2+\frac{w_2}{\sqrt{2}}\left(w_1+\frac{w_2}{\sqrt{2}}\right)+\frac{w_3}{\sqrt{3}}\left(w_1+\frac{w_2}{\sqrt{2}}+\frac{w_3}{\sqrt{3}}\right)\cdots\\
&\gtrsim w_1^2+\frac{w_2^2}{\sqrt{2}}\left(1+\frac{1}{\sqrt{2}}\right)+\frac{w_3^2}{\sqrt{3}}\left(1+\frac{1}{\sqrt{2}}+\frac{1}{\sqrt{3}}\right)+\cdots\\
&\gtrsim w_1^2+\cdots+w_d^2 = |\mbf{w}|^2 = |\mbf{v}|^2.
\end{align*}
Thus
\[\sup_{g\in G}|\langle\mbf{v},\rho(g)\mbf{a}\rangle|\gtrsim\frac{|\mbf{v}|}{\sqrt{\log d}}.\]

Finally, when $\mbf{v}\in\CC^d$, we can reduce to the real case.
Assume without loss of generality (since $\rho$ is real) that $|\on{Re}\mbf{v}|\ge|\on{Im}\mbf{v}|$. Then
\[\sup_{g\in G}|\langle\mbf{v},\rho(g)\mbf{a}\rangle|\ge\sup_{g\in G}|\langle\on{Re}\mbf{v},\rho(g)\mbf{a}\rangle|\gtrsim\frac{|\on{Re}\mbf{v}|}{\sqrt{\log d}}\ge\frac{|\mbf{v}|}{\sqrt{2\log d}}. \qedhere \]
\end{proof}
\begin{proof}[Proof of \cref{thm:main-signed-weighted}]
Let $\mbf{a}$ be as in \cref{lem:sup-inequality} (viewed as a column vector). Define $f \colon G \to \RR$ via
\[f(g) = d\sang{\rho(g),\mbf{a}\mbf{a}^\intercal}_\HS.\]
Therefore, by Fourier inversion, $\wh{f}$ is supported at $\rho$ and $\wh{f}(\rho) = \mbf{a}\mbf{a}^\intercal$. That is, $f\in V_\rho$ (as in \cref{eq:Vrho}).

As in \cref{eq:Vrhov}, let $V_{\rho,\mbf{a}}$ denote the subspace of $L^2(G)$ consisting of all $x \in L^2(G)$ of the form $x(g) = d\sang{\rho(g), \mbf{a}\mbf{v}^\dagger}_\HS$ for some $\mbf{v} \in \CC^d$. 
From the discussions in \cref{subsec:eigendecomposition}, we see that $M_f$ has exactly one nonzero eigenvalue, namely $1$, and its eigenspace is $V_{\rho,\mbf{a}}$.

We claim that
\[\|x\|_{L^\infty(G)} \gtrsim \sqrt{\frac{d}{\log d}} \|x\|_{L^2(G)}\]
for all $x \in V_{\rho, \mbf{a}}$. Letting $x(g) = d\sang{\rho(g), \mbf{a}\mbf{v}^\dagger}_\HS$, we have
\[\|x\|_{L^\infty(G)} = d \sup_{g \in G} | \sang{\rho(g), \mbf{a}\mbf{v}^\dagger}_\HS| = d\sup_{g \in G} | \sang{\mbf{v},\rho(g) \mbf{a}}| \gtrsim \frac{d}{\sqrt{\log d}} |\mbf{v}|,\]
where the last inequality is \cref{lem:sup-inequality}. Furthermore, by Parseval,
\[\|x\|_{L^2(G)} = \sqrt{d}\norm{\wh{x}(\rho)}_\HS = \sqrt{d}\snorm{\mbf{a}\mbf{v}^\dagger}_\HS = \sqrt{d}|\mbf{v}|.\]
Since $\abs{G} = 2^dd!$, we have $d = (1+o(1))\log{\abs{G}}/\log\log{\abs{G}}$, which completes the proof.
\end{proof}

\subsection{Unweighted construction}\label{subsec:unweighted}

Let us first explain the setup for this entire subsection.
Let $G$ be any finite group (later on we will specialize to $G = S_d \ltimes (\ZZ/2\ZZ)^d$).
Let $\rho \in \wh G$ be a real irreducible representation of $G$ of dimension $d = d_\rho$, i.e., $\rho \colon G \to O(d)$ is a homomorphism.
Finally, we assume that $\sqrt{\abs{G}/\log\abs{G}} > 15 d$.

Fix a unit vector $\mbf{a}\in\RR^d$ and let
\begin{equation}\label{eq:f}
f(g) = \frac{1 - \mbf{a}^\intercal\rho(g)\mbf{a}}{2}.
\end{equation}
By Cauchy--Schwarz, $f(g)\in [0,1]$. Furthermore, $f(g^{-1}) = f(g)$ since $\rho(g)$ has real entries.

Note that $f$ is quite similar to the example given in the proof of \cref{thm:main-signed-weighted}, but shifted and scaled so that its values lie in $[0,1]$. The idea is to sample a random Cayley graph from $f$. Then its eigenvalues will be close to the original. Furthermore, we will show that the top nontrivial eigenspace (which corresponds to $\rho$) does not change much, so the estimate \cref{lem:sup-inequality} will remain valid.

We now sample a random function $h$ based on $f$. Let $G'$ be the set of elements $g\in G$ with $g = g^{-1}$ and $G''$ be a subset of $G\setminus G'$ containing exactly one element of each set $\{g,g^{-1}\}\in\binom{G}{2}$ such that $g\neq g^{-1}$. For $g\in G'$, let $h(g)$ be $1$ with probability $f(g)$ and zero otherwise. For $g \in G''$, let $h(g) = h(g^{-1}) = 1$ with probability $f(g)$ and $0$ otherwise. The choices are independent across $G'\cup G''$. Note that $f(e) = 0$, so $h(e) = 0$.

For each $g \in G$, let $P_g$ denote the matrix with columns and columns indexed by $G \times G$ with entry $1/\abs{G}$ at position $(gx,x)$ for each $x \in G$ and zero elsewhere.
Viewing $M_f$ and $M_h$ as matrices (as described in \cref{subsec:eigendecomposition}), we have
\begin{align*}
M_h &= \sum_{g\in G'}h(g)P_g + \sum_{g\in G''}h(g)(P_g+P_{g^{-1}})\\
&= \sum_{g\in G'}(h(g)-f(g))P_g + \sum_{g\in G''}(h(g)-f(g))(P_g + P_{g^{-1}}) + \sum_{g\in G} f(g)P_g.
\end{align*}
We first compute the spectrum of $\EE[M_h] = M_f$.
\begin{lemma}\label{lem:spectrum}
Let $h$ be as above. The spectrum of $\mb{E}[M_h]$ is $1/2$ with multiplicity $1$, $-1/(2d)$ with multiplicity $d$, and $0$ with multiplicity $\abs{G}-d-1$.
\end{lemma}
\begin{proof}
By linearity of expectation, $\EE[M_h] = M_f$. Furthermore, since
\[f(g) = \frac{1 - \mbf{a}^\intercal\rho(g)\mbf{a}}{2} = \frac{1}{2} - \frac{1}{2d}(d\sang{\rho(g),\mbf{a}\mbf{a}^\intercal}_\HS),\]
we see that $\wh{f}(\on{triv}_G) = 1/2$ and $\wh{f}(\rho) = -\mbf{a}\mbf{a}^\intercal/(2d)$ by Fourier inversion, with $\wh{f}$ only supported at these two representations (here $\on{triv}_G$ is the trivial representation of $G$).

The analysis in \cref{subsec:eigendecomposition} therefore shows that $M_f$ has eigenvalues $1/2$ with multiplicity $1$, $-1/(2d)$ with multiplicity $d$, and $0$ for the rest.
\end{proof}

In order to establish quantitative concentration bounds regarding our sampling, we use the matrix Hoeffding inequality by Tropp  \cite{Tro12}. Recall that for self-adjoint $A,B$, the notation $A\preceq B$ means that $B-A$ is positive semidefinite.
\begin{theorem}[{\cite[Theorem~1.3]{Tro12}}]\label{thm:Hoeff}
Consider a finite sequence $\{ \mbf{X}_k \}$ of independent, random, self-adjoint matrices with dimension $d$, and let $\{ \mbf{A}_k \}$ be a sequence of fixed self-adjoint matrices.  Assume that each random matrix satisfies
\[
\mb{E} \mbf{X}_k = \mbf{0}
\quad\text{and}\quad
\mbf{X}_k^2 \preceq \mbf{A}_k^2
\quad\text{almost surely}.
\]
Then, for all $t \geq 0$,
\[
\mb{P}\bigg({ \lambda_{\max}\left( \sum_k \mbf{X}_k \right) \geq t \bigg)}
	\leq d \cdot e^{-t^2 / (8\sigma^2) }
	\quad\text{where}\quad
	\sigma^2 := \norm{ \sum_k \mbf{A}_k^2 }_{\on{op}}.
\]
\end{theorem}
\begin{remark}
The constant $8$ can be replaced by $2$ if $\mbf{X}_k$ and $\mbf{A}_k$ commute almost surely, which will hold true in our application. See \cite[Remark~7.4]{Tro12}.
\end{remark}
\begin{lemma}\label{lem:norm1}
Let $h$ be as above. Then
\[\mb{P}\left[\norm{M_h-\EE[M_h]}_{\on{op}}\le 4\sqrt{\frac{\log(6\abs{G})}{\abs{G}}}\right]\ge \frac23.\]
\end{lemma}
\begin{proof}
Note that when $g = g^{-1}$ we have that $P_g^2 = I/\abs{G}^2$. Otherwise note that 
\[(P_g + P_{g^{-1}})^2 = 2I/\abs{G}^2 + P_{g^2} + P_{g^{-2}} \preceq 4I/\abs{G}^2\]
as 
\[-2I/\abs{G}^2 + P_{g^2} + P_{g^{-2}} = (P_g - P_{g^{-1}})^2\preceq \mbf{0}.\]
Here we are using that $P_g-P_{g^{-1}}$ is antisymmetric.

Using that $|h(g)-f(g)|\le 1$ almost surely, and applying \cref{thm:Hoeff} to
\[M_h - \EE[M_h] = \sum_{g\in G'}(h(g)-f(g))P_g + \sum_{g\in G''}(h(g)-f(g))(P_g + P_{g^{-1}})\]
we find that
\[\mb{P}\bigg(\lambda_{\max}\big(M_h-\EE[M_h]\big)\ge t\bigg)\le \abs{G}\exp\bigg(\frac{-t^2\abs{G}}{16}\bigg)\]
for all $t\ge 0$. 
Applying the same inequality to $-M_h$ yields
\[\mb{P}\bigg(\norm{M_h-\EE[M_h]}_{\on{op}}\ge t\bigg)\le 2\abs{G}\exp\bigg(\frac{-t^2\abs{G}}{16}\bigg).\]
Setting $t = 4\abs{G}^{-1/2}(\log 6\abs{G})^{1/2}$ yields the lemma.
\end{proof}
This allows us to control the spectrum of $M_h$.
\begin{corollary}\label{cor:simple}
With $h$ as above, we have that with probability at least $2/3$ the number of eigenvalues of $M_h$ in $[-1/d,-1/(3d)]$ is exactly $d$.
\end{corollary}
\begin{proof}
This is an immediate consequence of Weyl's inequality on deviation of eigenvalues along with \cref{lem:spectrum,lem:norm1}.

More specifically, we have that $M_h$ and $\EE[M_h]$ are self-adjoint and $\norm{M_h-\EE[M_h]}_{\on{op}}\le 4\sqrt{\log(6\abs{G})/\abs{G}}$ with probability at least $2/3$. Thus, writing $\nu_1\ge\cdots\ge\nu_{\abs{G}}$ for the eigenvalues of $\EE[M_h]$ (which we know to be $1/2,0,\ldots,0,-1/(2d),\ldots,-1/(2d)$ with $d$ copies of $-1/(2d)$ by \cref{lem:spectrum}) and $\mu_1\ge\cdots\ge\mu_{\abs{G}}$ for the eigenvalues of $M_h$, we have
\[
|\mu_j-\nu_j|\le 4\sqrt{\frac{\log(6\abs{G})}{\abs{G}}} \le \frac{1}{6d}
\]
for all $1\le j\le \abs{G}$ by Weyl's inequality. 
The final inequality uses the assumption that $\sqrt{\abs{G}/\log \abs{G}} > 15 d$.
\end{proof}
We now show that $\wh{h}(\rho)$ and $\wh{f}(\rho) = \EE[\wh{h}(\rho)]$ are close.
\begin{lemma}\label{lem:norm2}
Let $h$ and $\rho$ be as above. Then
\[\PP\left[ \snorm{\wh{h}(\rho) - \mb{E}[\wh{h}(\rho)]}_{\on{op}} \le 4\sqrt{\frac{\log (6d)}{\abs{G}}} \right]\ge \frac23.\]
\end{lemma}
\begin{proof}
The proof is essentially identical to the proof of \cref{lem:norm1}. Note that if $g = g^{-1}$ then $\rho(g)^2=I_d$. Otherwise
\[(\rho(g)+\rho(g^{-1}))^2 = 2I_d + \rho(g^2) +  \rho(g^{-2})\preceq 4I_d\]
as 
\[(\rho(g)-\rho(g^{-1}))^2 \preceq \mbf{0}.\]
Here we are using that $\rho(g)-\rho(g^{-1}) = \rho(g)-\rho(g)^\intercal$ is antisymmetric.

Then, using the matrix Hoeffding bound \cref{thm:Hoeff}, it immediately follows that 
\[\mb{P}\bigg(\norm{\wh{h}(\rho) - \EE[\wh{h}(\rho)]}_{\on{op}}\ge t\bigg)\le 2d \exp\bigg(\frac{-t^2\abs{G}}{16}\bigg).\]
Setting $t = 4\abs{G}^{-1/2}(\log 6\abs{G})^{1/2}$ yields the lemma.
\end{proof}

We now show that the top eigenvector of $\wh{h}(\rho)$ and $\wh{f}(\rho) = \mb{E}[\wh{h}(\rho)]$ are close.
This is a special case of the Davis--Kahan Theorem \cite{DK70}.
We include a proof (adapted from \cite[Theorem~5.9]{RH19}) for completeness.
Recall that $\wh{h}(\rho)$ is real so its eigenvectors are real.

\begin{lemma}\label{lem:dk}
Let $h, \rho, \mbf{a}$ be as above. Let $\mbf{b}$ be a real unit eigenvector of the top eigenvalue of $\wh{h}(\rho)$. 
With probability at least $2/3$ we have
\[
\min \{|\mbf{a}+\mbf{b}|,|\mbf{a}-\mbf{b}|\} \le 16\sqrt{2}d\sqrt{\frac{\log (6d)}{\abs{G}}}.
\]
\end{lemma}
\begin{proof}
Let
\[\Sigma = \mb{E}[\wh{h}(\rho)] = \wh{f}(\rho) = -\frac{\mbf{a}\mbf{a}^\intercal}{2d}
\]
and
\[
\Sigma' = \wh{h}(\rho).
\]
Then we have that
\begin{align*}
\mbf{a}^\intercal\Sigma\mbf{a}-\mbf{b}^\intercal\Sigma\mbf{b}& = \mbf{a}^\intercal\Sigma'\mbf{a}-\mbf{b}^\intercal\Sigma\mbf{b} - \mbf{a}^\intercal(\Sigma'-\Sigma)\mbf{a}\\
& \le \mbf{b}^\intercal\Sigma'\mbf{b}-\mbf{b}^\intercal\Sigma\mbf{b} - \mbf{a}^\intercal(\Sigma'-\Sigma)\mbf{a}\\
& = \langle \Sigma-\Sigma', \mbf{a}\mbf{a}^\intercal-\mbf{b}\mbf{b}^\intercal\rangle_\HS\\
& \le \norm{\Sigma'-\Sigma}_{\on{op}}\cdot||\mbf{a}\mbf{a}^\intercal-\mbf{b}\mbf{b}^\intercal||_{S_1}\\
& \le \sqrt{2}\norm{\Sigma'-\Sigma}_{\on{op}}\cdot\norm{\mbf{a}\mbf{a}^\intercal-\mbf{b}\mbf{b}^\intercal}_\HS.
\end{align*}
The second inequality is an application of \cref{eq:Hold}, noting that $\norm{\cdot}_{\on{op}} = \norm{\cdot}_{\on{S_\infty}}$. The last step uses that for a matrix $M$ of rank at most 2, one has $\norm{M}_{S_1}\le\sqrt{2}\norm{M}_{S_2} = \sqrt{2}\norm{M}_\HS$.

Furthermore we have
\[\mbf{a}^\intercal\Sigma\mbf{a}-\mbf{b}^\intercal\Sigma\mbf{b} = \frac{1 - (\mbf{b}\cdot \mbf{a})^2}{2d}\]
and 
\[\norm{\mbf{a}\mbf{a}^\intercal-\mbf{b}\mbf{b}^\intercal}_\HS^2 = \on{Tr}((\mbf{a}\mbf{a}^\intercal-\mbf{b}\mbf{b}^\intercal)^2) = 2 - 2(\mbf{b}\cdot \mbf{a})^2.\]
Therefore we have that
\[(1-(\mbf{b}\cdot \mbf{a})^2)^{1/2}\le 4d\norm{\Sigma'-\Sigma}_{\on{op}}.\]
Then
\[(\mbf{a}\cdot\mbf{b})^2\ge 1 - 16d^2\norm{\Sigma'-\Sigma}_{\on{op}}^2.\]
Negating $\mbf{b}$ if necessary so that $\mbf{a} \cdot \mbf{b} \ge 0$, we have
\[|\mbf{a}-\mbf{b}|^2 = 2 - 2(\mbf{a}\cdot\mbf{b})\le 2\paren{1-\sqrt{1-16d^2\norm{\Sigma'-\Sigma}_{\on{op}}^2}}\le 32d^2\norm{\Sigma'-\Sigma}_{\on{op}}^2.\]
By \cref{lem:norm2}, with probability at least $3/4$, one has $\snorm{\Sigma'-\Sigma}_{\on{op}}^2  \le 16 \log (6d)/\abs{G}$, and the lemma follows.
\end{proof}
We combine the concentration results derived so far.
\begin{proposition}\label{prop:sampling}
Let $h, \mbf{a}$ be defined as above. For $\abs{G}$ sufficiently large, we have with probability at least $1/3$ that all of the following hold:
\begin{itemize}
    \item $M_h$ has exactly $d$ eigenvalues in the interval $[-1/d,-1/(3d)]$.
    \item $\wh{h}(\rho)$ has exactly one eigenvalue $\lambda$ in $[-1/d,-1/(3d)]$.
    \item There is a real unit eigenvector $\mbf{b}$ of $\wh{h}(\rho)$ of eigenvalue $\lambda$ with
    \[|\mbf{a}-\mbf{b}|\le 16\sqrt{2}d\sqrt{\frac{\log (6d)}{\abs{G}}}.\]
\end{itemize}
\end{proposition}
\begin{proof}
This is an immediate application of \cref{cor:simple,lem:norm2,lem:dk}. 
Note that although we are union-bounding over the failures of three statements (with failure rate at most $1/3$ each), the event used in \cref{lem:dk} is precisely that of \cref{lem:norm2}.
\end{proof}
We are now in position to prove \cref{thm:lower}. The proof will mimic that of \cref{thm:main-signed-weighted}.
\begin{proof}[Proof of \cref{thm:lower}]
As in \cref{subsec:weighted}, 
let $G = S_d\ltimes(\ZZ/2\ZZ)^d$ and $\rho$ be its standard representation on $\CC^d$ (permuting and negating coordinates), which is easily seen to be real. Furthermore let $\mbf{a}$ be the unit vector in the direction of $(1,1/\sqrt{2},\ldots,1/\sqrt{d})$, viewed as a column.

We sample $h$ as in the beginning of this subsection (the assumption $\sqrt{\abs{G}/\log\abs{G}} > 15 d$ holds for sufficiently large $d$). Let our graph be the Cayley graph with adjacency matrix $\abs{G}M_h$. As we care only about scale-invariant properties of eigenspaces, we restrict attention to $M_h$, which acts on $V = \CC^G$.

By \cref{prop:sampling}, for $\abs{G}$ large enough, with probability at least $2/3$ there are exactly $d$ eigenvalues of $M_h$ in $[-1/d,-1/(3d)]$, and $\wh{h}(\rho)$ has one eigenvalue $\lambda$ in this range with a unit eigenvector $\mbf{b}$ satisfying
\begin{equation}\label{eq:ab}
|\mbf{a}-\mbf{b}|\le 16\sqrt{2}d\sqrt{\frac{\log (6d)}{\abs{G}}}.
\end{equation}

By the characterization of eigenspaces of Cayley graphs in \cref{subsec:eigendecomposition}, we see that $\wh{h}(\rho)$ contributes a $d$-dimensional eigenspace to $M_h$ for each of its eigenvalues. Therefore we see that the $d$ eigenvalues of $M_h$ in $[-1/d,-1/(3d)]$ are precisely $d$ copies of this eigenvalue $\lambda$.

In particular, $M_h$ has an eigenvalue $\lambda$ which has eigenspace precisely $V_{\rho,\mbf{b}}$, which recall from \cref{eq:Vrhov} is
\[V_{\rho,\mbf{b}} = \{x\in L^2(G): x(g)=d\sang{\rho(g),\mbf{b}\mbf{w}^\dagger}_\HS \text{ for some } \mbf{w}\in\CC^d\}.\]

Now to show the construction satisfies the conclusion of the theorem. Let $x \in V_{\rho,\mbf{b}}$. 
We wish to show that (recall $\abs{G} = 2^dd!$ so that $d= (1+o(1))\log{\abs{G}}/\log\log{\abs{G}}$)
\[
\norm{x}_{L^\infty(G)} \gtrsim \sqrt{\frac{d}{\log d}} \norm{x}_{L^2(G)} 
= (1+o(1))
\frac{\sqrt{\log\abs{G}}}{\log\log\abs{G}} \norm{x}_{L^2(G)}.
\]
If $x(g) = d\ang{\rho(g),\mbf{b}\mbf{v}^\dagger}_\HS = d\mbf{v}^\dagger\rho(g)\mbf{b}$ for some $\mbf{v}\in\CC^d$, then by Parseval,
\[\|x\|_{L^2(G)} = \sqrt{d}\snorm{\mbf{b}\mbf{v}^\dagger}_\HS = \sqrt{d}|\mbf{v}|.\]
Furthermore,
\begin{align*}
\|x\|_{L^\infty(G)} 
&= \sup_{g\in G}|d\langle\rho(g),\mbf{b}\mbf{v}^\dagger\rangle_\HS| \\
&\ge 
\sup_{g\in G}|d\langle\rho(g),\mbf{a}\mbf{v}^\dagger\rangle_\HS|-\sup_{g\in G}|d\langle\rho(g),(\mbf{a}-\mbf{b})\mbf{v}^\dagger\rangle_\HS|\\
&\gtrsim\frac{d|\mbf{v}|}{\sqrt{\log d}}-d|\mbf{a}-\mbf{b}||\mbf{v}|
\\
&\gtrsim 
\sqrt{\frac{d}{\log d}}\|x\|_{L^2(G)}.
\end{align*}
for sufficiently large $d$, by \cref{lem:sup-inequality} and \cref{eq:ab}. 
\end{proof}

\begin{remark}
To produce the graphs in \cref{fig:spectral-drawing}, we produce a Cayley graph on $G = S_d\ltimes(\ZZ/2\ZZ)^d$ with each possible generator $g\neq e$ included with probability $C(1-f(g))$. ($C = 2/3$ for the first figure and $C = 1/7$ for the second; these constants merely serve to sparsify the graphs for aesthetic purposes.)

In the proof of \cref{thm:lower} above, we sample a graph via a similar procedure and deduce that with positive probability it has an unbounded eigenspace of dimension $d$. In fact, we can further deduce that this unbounded eigenspace has negative eigenvalue, and is the largest eigenvalue in magnitude after the trivial eigenvalue. Thus, a similar proof shows that the graph we sampled above, with positive probability, has an unbounded eigenspace of multiplicity $d\ge 2$ which contains the second and third eigenvalues. Therefore any possible spectral drawing of such a graph will have width at least $c\sqrt{\log\abs{G}}/{\log\log\abs{G}}$, as required.

Finally, in practice, we only sampled the small values $d = 3$ and $d = 4$. In this situation there is a decent probability of sampling a graph not satisfying the desired properties, namely of having the second and third eigenvalues come from the standard representation $\rho$ of $G$, and that their eigenspace is precisely $d$-dimensional. To produce \cref{fig:spectral-drawing}, we check for these properties and resample until they hold.
\end{remark}

\section{Upper Bound}\label{sec:converse}

In this section we prove \cref{thm:upper-cayley,thm:upper-transitive}, showing that all Cayley graphs (\cref{subsec:upper-cayley}) and transitive graphs (\cref{subsec:upper-vertex-transitive}) on $n$ vertices have an $O(\sqrt{\log n})$-bounded orthonormal eigenbasis. 

\subsection{Cayley graphs} \label{subsec:upper-cayley}

\begin{lemma}\label{lem:projection}
Given a set $S$ of $n$ unit vectors in $\mb{R}^d$ (resp. $\CC^d$) we can find an orthonormal (resp. unitary) basis $L$ of $\mb{R}^d$ (resp. $\CC^d$) such that 
\[\max_{\mbf{w}\in L, \mbf{v}\in S}\abs{\ang{\mbf{w},\mbf{v}}}\lesssim \sqrt{\frac{\log (dn)}{d}}.\]
Furthermore, when $S\subseteq\CC^{d}$ we can choose $L$ to have all real vectors.
\end{lemma}
\begin{proof}
Let us first do the real case.
Recall the following standard bound on the volume of spherical caps in high dimensions (e.g., \cite[Lemma~2.2]{Bal97}): for a uniformly random unit vector $\mbf{w} \in \RR^d$ and fixed unit vector $\mbf{v} \in \RR^d$, one has
\begin{equation}\label{eq:spherical-cap}
\PP( \ang{\mbf{w}, \mbf{v}} \ge \epsilon )
= \PP( \abs{\mbf{w} - \mbf{v}}^2 \le  2-2\epsilon )
\le e^{-d \epsilon^2/2}.
\end{equation}

Let $\epsilon = \sqrt{2\log (4dn)/d}$. Then applying union bound with \cref{eq:spherical-cap}, we find that an orthonormal basis $L = \{\mbf{w}_1,\ldots,\mbf{w}_d\}$ uniformly at random satisfies
\[
\PP\paren{ \max_{\mbf{w} \in L, \mbf{w} \in S}\abs{\ang{\mbf{w}, \mbf{v}}} \ge \epsilon } \le 2dn e^{-d\epsilon^2/2} \le \frac12.
\]
So there is some $L$ such that $\ang{\mbf{w}, \mbf{v}} < \epsilon \lesssim \sqrt{\log (dn) /d}$ for all $\mbf{w} \in L$ and $\mbf{v} \in S$.

For the complex case, for each $v\in S$ write $\mbf{v} = \mbf{v}_1 + i \mbf{v}_2$ for $\mbf{v}_1,\mbf{v}_2\in \mb{R}^n$. We apply the real case of this lemma to the set
\[S' = \bigg\{\frac{\mbf{v}_1}{\snorm{\mbf{v}_1}_2}: \mbf{v}\in S\bigg\}\cup\bigg\{\frac{\mbf{v}_2}{\snorm{\mbf{v}_2}_2}: \mbf{v}\in S\bigg\}\]
to obtain a basis $L$ of $\mb{R}^d\subseteq\mb{C}^d$. Then for each $\mbf{w}\in L$ and $\mbf{v}\in S$, we have
\[\abs{\ang{\mbf{w},\mbf{v}}}\le\abs{\ang{\mbf{w},\mbf{v}_1}} + \abs{\ang{\mbf{w},\mbf{v}_2}}\le\abs{\ang{\mbf{w},\frac{\mbf{v}_1}{\snorm{\mbf{v}_1}_2}}} + \abs{\ang{\mbf{w},\frac{\mbf{v}_2}{\snorm{\mbf{v}_2}_2}}}\lesssim\sqrt{\frac{\log(dn)}{d}}.\qedhere\]
\end{proof}

Now we are ready to prove \cref{thm:upper-cayley}. In essence our argument amounts to choosing a random unitary basis (via \cref{lem:projection}) for each eigenspace coming from the representation theory of the group $G$ of the Cayley graph. However, in order to choose a \emph{real} orthonormal eigenbasis, we essentially pair up conjugate irreducible representations. 
This technicality is unnecessary if we only wish to find a unitary eigenbasis.

\begin{proof}[Proof of \cref{thm:upper-cayley}]
Let $S$ be the symmetric generating set of the Cayley graph and $\mbm{1}_S$ be the corresponding indicator function. Recall, from \cref{subsec:eigendecomposition}, the orthogonal decomposition of $V = \CC^G$ as
\begin{equation}\label{eq:V-decomp}
V 
=
\bigoplus_{\rho \in \wh G} V_\rho 
=
\bigoplus_{\rho\in\wh{G}}\bigoplus_{j=1}^{d_\rho}V_{\rho,\mbf{v}_j^\rho}.
\end{equation}
See \cref{subsec:eigendecomposition} for the definitions of $V_\rho$ and $V_{\rho,\mbf{v}}$. 
Here the vectors $\mbf{v}_j^\rho$, $j \in [d_\rho]$, form a unitary eigenbasis of $\wh{\mbm{1}_S}(\rho)$ with respective eigenvalues $\lambda_{\rho,j}$.
The eigenspace of $M_{\mbm{1}_S}$ corresponding to some eigenvalue $\lambda$ is a direct sum of all components $V_{\rho,\mbf{v}_j^\rho}$ with $\lambda_{\rho, j} = \lambda$.
Using this decomposition, we shall construct a real-valued unitary eigenbasis for the operator $M_{\mbm{1}_S}$ on $V$ as follows:
\begin{itemize}
    \item For each conjugate pair $(\rho, \ol\rho)$ and eigenvector $\mbf{b}$ of $\wh{\mbm{1}_S}(\rho)$,
    we will find a real-valued unitary basis of $V_{\rho, \mbf{b}} \oplus V_{\ol\rho, \ol{\mbf{b}}}$.
    \item For each irreducible representation $\rho \in \wh G$ such that $\rho$ and $\ol{\rho}$ are isomorphic (such representations are called \emph{self-dual}), we will find a special eigenbasis of $\wh{\mbm{1}_S}(\rho)$, and a specific construction giving a unitary basis of $V_\rho$ that will depend on whether the matrix $Q$ satisfying $\rho Q = Q \ol{\rho}$ is symmetric or antisymmetric.
\end{itemize}
Note that the second case includes $\rho$ which can be realized as a real representation, but not all self-dual representations are of this form (e.g. the two-dimensional irreducible representation of the quaternion group $Q_8$). Furthermore, we will ensure that all the functions $x$ chosen as basis elements above satisfy 
\[
\norm{x}_{L^\infty(G)}
\lesssim\sqrt{\log \abs{G}}\norm{x}_{L^2(G)}.
\]

As a model case we consider $\rho \in \wh G$ a real irreducible representation of dimension $d = d_\rho$, acting on $\RR^d$.
We will not, strictly speaking, need this analysis in the final argument as our treatment of self-dual representations is strictly more general.
Let $\mbf{b} \in \RR^d$ be a real eigenvector of $\wh{\mbm{1}_S}(\rho)$. 
Recall
\[
V_{\rho,\mbf{b}} = \{x\in L^2(G): x(g) = d\sang{\rho(g),\mbf{b}\mbf{w}^\dagger}_\HS\text{ for some }\mbf{w}\in \CC^d\}.
\]
Similar to in the proof of \cref{thm:lower}, writing $x_{\mbf{v}}: G\to\CC$ for the function $x_{\mbf{v}}(g) = d\sang{\rho(g),\mbf{b}\mbf{v}^\dagger}_\HS  = d \sang{\rho(g)\mbf{v},\mbf{b}}$, we have
\[\|x_{\mbf{v}}\|_{L^\infty(G)} = d \cdot \sup_{g\in G}|\langle\mbf{v},\rho(g)\mbf{b}\rangle|\]
and
\[\|x_{\mbf{v}}\|_{L^2(G)} = \sqrt{d}|\mbf{b}||\mbf{v}| = \sqrt{d}|\mbf{v}|.\]
By the real version of \cref{lem:projection} applied to $S = \{\rho(g)\mbf{b}: g\in G\}$, there is an orthonormal basis $L$ of $\RR^d$ such that (note that $d^2\le\abs{G}$)
\[\sup_{\mbf{v}\in L, g\in G}|\langle\mbf{v},\rho(g)\mbf{b}\rangle|\lesssim\sqrt{\frac{\log (d\abs{G})}{d}}
\lesssim \sqrt{\frac{\log \abs{G}}{d}}. \]
Then for each $\mbf{v}\in L$ we have
\[\|x_{\mbf{v}}\|_{L^\infty(G)}
\lesssim\sqrt{\log \abs{G}}\|x_{\mbf{v}}\|_{L^2(G)}.\]
By Parseval's identity, we see that $\{x_{\mbf{v}}/\sqrt{d} : \mbf{v}\in L\}$ forms a unitary basis of $V_{\rho,\mbf{b}}$. Also note that $x_{\mbf{v}}$ is real-valued for each $\mbf{v}\in\RR^d$. This completes the case of real $\rho$.

Next, let $(\rho,\ol{\rho})$ be a conjugate pair of irreducible representations with $\rho$ and $\ol{\rho}$ not isomorphic to each other.
Again let $d = d_\rho$. 
For each eigenvector $\mbf{b} \in \CC^d$ of $\wh{\mbm{1}_S}(\rho)$, 
by the complex version of \cref{lem:projection},
we find a unitary basis $L$ of $\CC^d$ so that
$\sup_{\mbf{v} \in L, g \in G} \abs{\ang{\mbf{v}, \rho(g)\mbf{b}}} \lesssim \sqrt{(\log \abs{G})/d}$. 
Again writing $x_{\mbf{v}}(g) = d\sang{\rho(g),\mbf{b}\mbf{v}^\dagger}$, we find that $\{x_{\mbf{v}}/\sqrt{d}: \mbf{v}\in L\}$ is a unitary basis of $V_{\rho,\mbf{b}}$ and
\[\|x_{\mbf{v}}\|_{L^\infty(G)}\lesssim\sqrt{\log\abs{G}}\|x_{\mbf{v}}\|_{L^2(G)}.\]
Likewise, taking conjugates, we see that 
$\{\ol{x_{\mbf{v}}}/\sqrt{d}: \mbf{v}\in L\}$
is a unitary basis of $V_{\ol \rho, \ol{\mbf{b}}}$.
Recall that if $M_{\mbm{1}_S} x_{\mbf{v}} = \lambda x_{\mbf{v}}$
then $M_{\mbm{1}_S} \ol{x_{\mbf{v}}} = \lambda \ol{ x_{\mbf{v}}}$ as $M_{\mbm{1}_S}$ is symmetric and hence $\lambda$ is real.

The collection of $2d$ vectors
\[
y_{\mbf{v}}^0 = \frac{x_{\mbf{v}}+\ol{x_{\mbf{v}}}}{\sqrt{2d}}
\quad \text{and} \quad 
y_{\mbf{v}}^1 = \frac{x_{\mbf{v}}-\ol{x_{\mbf{v}}}}{i\sqrt{2d}},
\]
as $\mbf{v}$ ranges over $L$, forms a real-valued unitary basis of 
$V_{\rho,\mbf{b}} \oplus V_{\ol \rho, \ol{\mbf{b}}}$.
Furthermore,
\[
\|y_{\mbf{v}}^0\|_{L^\infty(G)}
\le\sqrt{\frac{2}{d}}\norm{x_{\mbf{v}}}_{L^\infty(G)}
\lesssim \sqrt{\frac{\log\abs{G}}{d}}\norm{x_{\mbf{v}}}_{L^2(G)} = \sqrt{\log\abs{G}}\norm{y_{\mbf{v}}^0}_{L^2(G)}
\]
and similarly for $y_{\mbf{v}}^1$. 
This completes the case of non-self-dual complex irreducible representations.

Finally, let $\rho$ be a self-dual irreducible representation. Again let $d = d_\rho$. Note that $\wh{\mbm{1}_S}(\rho)$ is Hermitian, hence we can choose coordinates on the representation so that it is a real diagonal matrix. Having done so, we now note that $g\mapsto\rho(g)$ and $g\mapsto\ol{\rho(g)}$ are isomorphic representations on the same space, since $\rho$ is self-dual (where complex conjugation is done in the natural way with respect to the coordinates chosen on the space). Hence there is a unitary operator $Q$ so that
\[\rho(g)Q = Q\ol{\rho(g)}\]
for all $g\in G$. Thus
\[\rho(g)Q\ol{Q} = Q\ol{\rho(g)Q} = Q\ol{Q}\rho(g)\]
for all $g\in G$. By Schur's lemma, we deduce that
\[Q\ol{Q} = \omega I\]
for some $\omega\in\mb{C}$. Since $Q$ is unitary, $QQ^\dagger = I$ (we use ${}^\dagger$ to denote Hermitian transpose and ${}^\intercal$ for transpose), which yields
\[Q = \omega Q^\intercal = \omega^2 Q.\]
Since $Q$ is invertible, we deduce $\omega^2 = 1$, and hence $\omega\in\{\pm 1\}$.

From $\wh{\mbm{1}_S}(\rho) = \sum_{g \in S} \rho(g)/\abs{G}$ and $\rho(g)Q = Q\ol{\rho(g)}$
we obtain
\[Q^\dagger\wh{\mbm{1}_S}(\rho)Q = \ol{\wh{\mbm{1}_S}(\rho)} = \wh{\mbm{1}_S}(\rho),\]
since $\wh{\mbm{1}_S}(\rho)$ is a real diagonal matrix, as noted earlier. Therefore $Q$ and $\wh{\mbm{1}_S}(\rho)$ commute. Now choose a unitary simultaneous eigenbasis of $Q$ and $\wh{\mbm{1}_S}(\rho)$, which can be done by the spectral theorem as both operators are normal.

We will actually take a specific basis with more structure. First, note that if $\mbf{b}$ is an eigenvector of $Q$ with eigenvalue $\lambda$ then
\[Q^{-1}\ol{\mbf{b}} = Q^\dagger\ol{\mbf{b}} = \ol{Q^T\mbf{b}} = \omega^{-1}\ol{Q\mbf{b}} = \omega^{-1}\ol{\lambda \mbf{b}}.\]
Thus $\ol{\mbf{b}}$ is an eigenvector of $Q$ with eigenvalue $(\omega^{-1}\ol{\lambda})^{-1} = \omega\lambda$, using $|\lambda| = 1$ since $Q$ is unitary. 
Now we break into sub-cases depending on the value of $\omega\in\{\pm 1\}$.

If $\omega = -1$, then we see that our unitary simultaneous eigenbasis of $Q$ and $\wh{\mbm{1}_S}(\rho)$ can be chosen so that if $\mbf{b}$ is in it, then so is $\ol {\mbf{b}}$, since $\mbf{b}$ and $\ol{\mbf{b}}$ lie in distinct orthogonal eigenspaces of $Q$. 
For such an eigenvector $\mbf{b}$, as earlier we can apply  \cref{lem:projection} to obtain a unitary basis $L$ of $\CC^d$ so that
$\sup_{\mbf{v} \in L, g \in G} \abs{\ang{\mbf{v}, \rho(g)\mbf{b}}} \lesssim \sqrt{(\log \abs{G})/d}$. 
Again writing $x_{\mbf{v}}(g) = d\sang{\rho(g),\mbf{b}\mbf{v}^\dagger}$, we find that $\{x_{\mbf{v}}/\sqrt{d}: \mbf{v}\in L\}$ is a unitary basis of $V_{\rho,\mbf{b}}$ and $\|x_{\mbf{v}}\|_{L^\infty(G)}\lesssim\sqrt{\log\abs{G}}\|x_{\mbf{v}}\|_{L^2(G)}$.

Note that $\ol{x_{\mbf{v}}} \in V_{\rho, \ol{\mbf{b}}}$ since
\[\ol{x_{\mbf{v}}}(g) = 
d\ol{\sang{\rho(g)\mbf{v},\mbf{b}}} 
= d\sang{\ol{\rho(g)\mbf{v}},\ol{\mbf{b}}} 
= d\sang{Q^\dagger\rho(g)Q\ol{\mbf{v}},\ol{\mbf{b}}} 
= d\sang{\rho(g)Q\ol{\mbf{v}},Q\ol{\mbf{b}}}
= d\sang{\rho(g)Q\ol{\mbf{v}}, \omega\lambda\ol{\mbf{b}}}.\]
This shows that $\ol{V_{\rho, \mbf{b}}} = V_{\rho, \ol{\mbf{b}}}$.
Furthermore, as $M_{\mbm{1}_S} x_{\mbf{v}} = \lambda x_{\mbf{v}}$, we have $M_{\mbm{1}_S} \ol{x_{\mbf{v}}} = \lambda \ol{x_{\mbf{v}}}$ since $\lambda$ is real.
Then, as $\mbf{v}$ varies over $L$, the functions
\[
y_{\mbf{v}}^0 = \frac{x_{\mbf{v}}+\ol{x_{\mbf{v}}}}{\sqrt{2d}}
\quad \text{and} \quad 
y_{\mbf{v}}^1 = \frac{x_{\mbf{v}}-\ol{x_{\mbf{v}}}}{i\sqrt{2d}}
\]
form a real-valued unitary basis of $V_{\rho, \mbf{b}} \oplus V_{\rho, \ol{\mbf{b}}}$. This completes the proof of the case $\omega = -1$.

If $\omega = 1$, then for every eigenvector $\mbf{b}$ of $Q$ with eigenvalue $\lambda$, $\ol{\mbf{b}}$ is another eigenvector of $Q$ also with the same eigenvalue $\lambda$.
Thus every eigenspace $U$ of $Q$ satisfies $U = \ol U$. 
A $\CC$-vector space $U$ satisfying $U = \ol U$ is always the $\CC$-extension of the $\RR$-vector space $\re U = \{(\re u_1, \dots, \re u_d) : (u_1, \dots, u_d) \in U\}$ (since every $\mbf{v} \in U$ can be written as $\mbf{x} + i \mbf{y}$ with $\mbf{x} = (\mbf{v} + \ol{\mbf{v}})/2$ and $\mbf{y} = (\mbf{v} - \ol{\mbf{v}})/(2i)$ both having real coordinates). 
Thus we can choose an orthonormal basis of $\RR^d$ consisting of real-valued eigenvectors $\mbf{b}$ of $Q$. 

Fix such a real eigenvector $\mbf{b}$ of $Q$.
For any $\mbf{v} \in \CC^d$, setting 
$x_{\mbf{v}}(g) = d\sang{\rho(g),\mbf{b}\mbf{v}^\dagger}$ 
as before, we have $\ol{x_{\mbf{v}}} \in V_{\rho, \mbf{b}}$ since
\[
\ol{x_{\mbf{v}}}(g) 
= d\ol{\sang{\rho(g)\mbf{v},\mbf{b}}} 
= d \sang{\ol{\rho(g)\mbf{v}},\mbf{b}} 
= d \sang{Q^\dagger\rho(g)Q\ol{\mbf{v}},\mbf{b}} 
= d \sang{\rho(g)Q\ol{\mbf{v}},Q\mbf{b}}
= d \sang{\rho(g)Q\ol{\mbf{v}},\lambda \mbf{b}}.
\]
Thus $\ol{V_{\rho, \mbf{b}}} = V_{\rho, \mbf{b}}$, which then must be the $\CC$-extension of the $d$-dimensional $\RR$-vector space $\re V_{\rho, \mbf{b}}$.

We now apply \cref{lem:projection} to find a unitary basis $L$ of the $d$-dimensional $\RR$-vector space $\{\mbf{v} \in \CC^d : x_{\mbf{v}} \in \re V_{\rho, \mbf{b}} \}$ satisfying
$\sup_{\mbf{v} \in L, g \in G} \abs{\ang{\mbf{v}, \rho(g)\mbf{b}}} \lesssim \sqrt{(\log \abs{G})/d}$. Then $\{x_\mbf{v}/\sqrt{d} : \mbf{v}\in L\}$ is a real-valued unitary basis of $V_{\rho, \mbf{b}}$ with $\|x_{\mbf{v}}\|_{L^\infty(G)}\lesssim\sqrt{\log\abs{G}}\|x_{\mbf{v}}\|_{L^2(G)}$.
\end{proof}

\subsection{Vertex-transitive graphs} \label{subsec:upper-vertex-transitive}
We now extend \cref{thm:upper-cayley} to vertex-transitive graphs; the idea is the same as before, except we first lift to a Cayley graph on the automorphism group $G$ of the original. This trick is closely related to the proof of \cite[Theorem~2.2]{CZ17}.
\begin{proof}[Proof of \cref{thm:upper-transitive}]
Let $G$ denote the automorphism group of the given vertex-transitive graph, acting on the vertex-set from the right. Fix a vertex as the root of the graph. 
Let $H$ denote the stabilizer of the root. Then the vertices of the graph are given by right cosets $Hg$, $g \in H\backslash G$, with the root corresponding to the trivial coset $H$. Thus $|H\backslash G| = n$.

Let $f \colon H\backslash G \to \CC$ denote the edge-weights from the root to other vertices of the graph. Since $G$ induces automorphisms on the graph, the edgeweight of $(H, Hg)$ equals that of $(Hh, Hgh) = (H, Hgh)$ for all $h \in H$. Hence $f(gh) = f(g)$ for all $g \in G$ and $h \in H$. So we can view $f$ as a function $f\colon G\to\CC$ that is $H$-invariant from both the left and right.

A function on the vertex set is represented as $x \colon H\backslash G \to \CC$, which we will view as a left-$H$-invariant $x \colon G \to \CC$, i.e., $x(hg) = x(g)$ for all $g \in G$ and $h \in H$.

A function $x \colon G \to \CC$ satisfies $x(hg) = x(g)$ for all $g \in G$ and $h \in H$ (i.e., it is left-$H$-invariant) 
if and only if 
$\wh x(\rho) = \rho(h)\wh x(\rho)$ for all $\rho \in \wh{G}$ and $h \in H$. 
The latter condition is equivalent to saying that 
the column-space of $\wh x(\rho)$ lies in $U_\rho$, the $1$-eigenspace of $\rho|_H$:
\[
U_\rho := \{v \in W_\rho: \rho(h) v = v\text{ for all } h \in H\}.
\]
The forward implication follows from the Fourier transform formula $\wh x(\rho) = \EE_{g \in G} x(g)\rho(g)$,
while the reverse implication follows from the inversion formula $x(g) = \sum_\rho d_\rho \on{tr}(\rho(g)^\dagger \wh x(\rho))$.

Let $m_\rho = \dim U_\rho$. By counting the dimension of the space of all left-$H$-invariant functions, we obtain
\begin{equation}\label{eq:transitive-dim-count}
\sum_{\rho \in \wh{G}} d_\rho m_\rho = \abs{H\backslash G} = n.
\end{equation}
Indeed, the condition that $\wh{x}(\rho)$ has column-space contained within $U_\rho$ restricts $\wh{x}(\rho)$ to a $d_\rho m_\rho$-dimensional subspace of $\on{End} W_\rho$.

Since $f \colon G \to \CC$ is both left- and right-$H$-invariant, $\wh{f}(\rho) = \rho(h) \wh{f}(\rho) \rho(h')$ for all $h,h' \in H$. So $\wh{f}(\rho)$ leaves $U_\rho$ invariant. Let $v^\rho_1, \dots, v^\rho_{m_\rho} \in U_\rho$ be an eigenbasis of the action of $\wh{f}(\rho)$ on $U_\rho$.

For each $\rho \in \wh{G}$, choose a unitary basis $a_1^\rho, \dots, a_{d_\rho}^\rho$ of $W_\rho$, and for each $j \in [d_\rho]$ and $k\in[m_\rho]$, define $x^\rho_{j,k} \colon G \to \CC$ by setting $\wh {x^\rho_{j,k}}(\rho) = v^\rho_k (a_j^\rho)^\dagger/\sqrt{d_\rho}$ and $\wh x (\rho') = 0$ for all $\rho' \ne \rho$. 

The functions $x^\rho_{j,k}$, with $\rho \in \wh{G}$, $j \in [d_\rho]$, $k \in [m_\rho]$ satisfy the following properties.
\begin{enumerate}
	\item $x^\rho_{j,k}$ is left-$H$-invariant, i.e., are functions $H\backslash G \to \CC$ (since the columns of $\wh{x^\rho_{j,k}}$ are in $U_\rho$),
	\item $x^\rho_{j,k}$ is an eigenfunctions of $M_f$ (since $\wh{f}(\rho)\wh{x^\rho_{j,k}}(\rho)$ is a scalar multiple of $\wh{x^\rho_{j,k}}(\rho)$)
	\item The functions $x^\rho_{j,k}$ are pairwise orthogonal and $\snorm{x^\rho_{j,k}}_2 = 1$. Indeed, by Parseval, one has $\sang{x^\rho_{j,k},x^{\rho'}_{j',k'}}_{\HS} = 0$ if $\rho \ne \rho'$, and 
	\[ 
	\sang{x^\rho_{j,k},x^\rho_{j',k'}}
	=
	d_\rho \sang{\wh{x^\rho_{j,k}}(\rho),\wh{x^\rho_{j',k'}}(\rho)}_{\HS}  = \sang{a_{j'}^\rho, a_j^\rho} \sang{v_k^\rho, v_{k'}^\rho} = 1_{j=j'} 1_{k=k'}.
	\]
	\item They form a basis of all functions $H\backslash G \to \CC$ (by orthogonality and dimension counting \cref{eq:transitive-dim-count}).
\end{enumerate}
Furthermore, we have for each $\rho$, $j\in [d_\rho]$, and $k\in[m_\rho]$ that
\[
x^\rho_{j,k}(g) = d_\rho \ang{\rho(g), \wh {x^\rho_{j,k}}(\rho)}_{\HS} 
= \sqrt{d_\rho} \ang{\rho(g), v_k^\rho(a_j^\rho)^\dagger}_{\HS} = 
\sqrt{d_\rho} \ang{\rho(g)a_j^\rho, v^\rho_k}.
\]

For each fixed $\rho \in \wh G$, set $S_\rho= \{\rho(g)v_k^\rho: g\in G, k\in[m_\rho]\}$. 
Now since $\rho(h)v_k^\rho = v_k^\rho$ for all $k\in[m_\rho]$ (because $M_\rho$ is the $1$-eigenspace of $\rho|_H$), we see that $|S| \le |G/H|m_\rho \le n^2$. 
By the complex version of \cref{lem:projection}, there exists a choice of the unitary basis $a_1^\rho, \dots, a_{d_\rho}^\rho$ in the definition of $x^\rho_{j,k}(g)$ earlier so that
\[
\sup_{j\in[d_\rho],k\in[m_\rho]}\sup_{g\in G}|x^\rho_{j,k}(g)|
= \sup_{j\in[d_\rho],k\in[m_\rho]} \sqrt{d_\rho} \ang{\rho(g)a_j^\rho, v^\rho_k}
\lesssim \sqrt{\log n}.
\]
Thus the functions $x^\rho_{j,k}$, ranging over all irreducible representations $\rho$ and indices $j \in [d_\rho]$ and $k \in [m_\rho]$, form a unitary $\sqrt{\log n}$-bounded eigenbasis.

To obtain a real orthonormal eigenbasis, we can repeat the technique in the proof of \cref{thm:upper-cayley} in the previous subsection. We omit the details.
\end{proof}

\section{Small-set expansion in random Cayley graphs}\label{sec:small-set-expansion}

In this section we prove \cref{thm:naor-nonabelian}, extending Naor's theorem~\cite{Nao12} on small-set expansion for random Cayley graphs to nonabelian groups. 
Recall the definition of the Schatten $p$-norm $\norm{\cdot}_{S_p}$ from \cref{sec:schatten}.
We state below a Hausdorff--Young inequality for groups, which is standard though we include its short proof (see \cite{Kun58} for a proof for locally compact unimodular groups).

\begin{lemma}[Hausdorff–-Young inequality for groups] \label{lem:hy}
Let $G$ be a finite group and $f \colon G \to \CC$. For any $1 \le p \le 2 \le q \le \infty$ with $1/p + 1/q = 1$, one has
\[
\norm{f}_{S_q} \le \norm{f}_{L^p(G)}.
\]
\end{lemma}

\begin{proof}
By the Riesz–-Thorin interpolation theorem, it suffices to check the inequality for $(p,q) = (2,2)$ and $(1,\infty)$. For $(p,q) = (2,2)$, we have $\|f\|_{S_2} = \|f\|_2$ by Parseval. For $(p,q) = (1,\infty)$ we have
\[
\norm{f}_{S_\infty} = \max_{\rho \in \wh{G}} \snorm{\wh{f}(\rho)}_{\on{op}} 
= \max_{\rho \in \wh{G}}\norm{\mathbb{E}_g f(g)\rho(g) }_{\on{op}} \le \mathbb{E}_{g \in G} \abs{f(g)} = \norm{f}_1
\]
as $\norm{\rho(g)}_{\on{op}} =1$ for all $g \in G$.
\end{proof}

\begin{lemma} \label{lem:qform-schatten}
Let $G$ be a finite group. For functions $f,x \colon G \to \CC$ and real $p \ge 1$, one has
\[
\abs{\ang{ x, f \ast x }} \le \norm{f}_{S_p} \norm{x}_{2p/(p+1)}^2.
\]
\end{lemma}

\begin{proof}
For each $\rho \in \wh{G}$ we have
\begin{align}
\abs{\ang{\wh x(\rho), \wh{f}(\rho) \wh x(\rho) ) }_\HS}
&= 
\abs{\Tr \paren{\wh x(\rho)^\dagger \wh{f}(\rho) \wh x(\rho)}}
= 
\abs{\Tr \paren{\wh x(\rho) \wh x(\rho)^\dagger \wh{f}(\rho))}} \nonumber
\\
&\le 
\snorm{\wh{f}(\rho)}_{S_p} \snorm{\wh x(\rho) \wh x(\rho)^\dagger}_{S_{p/(p-1)}}
=
\snorm{\wh{f}(\rho)}_{S_p} \snorm{\wh x(\rho) }_{S_{2p/(p-1)}}^2. \label{eq:qform-holder-step}
\end{align}
Here the inequality step uses the matrix H\"older inequality~\cref{eq:Hold}: $\Tr(AB) \le \norm{A}_{S_p} \norm{B}_{S_{p/(p-1)}}$ for all $p\in[1,\infty]$.
The last step uses that the singular values of a matrix $A$ are the square roots of the singular values of $AA^\dagger$.

Thus, applying the convolution and Parseval identities for the nonabelian Fourier transform (\cref{subsec:fourier}), we have
\begin{align*}
\abs{\ang{ x, f \ast x }}
&= \abs{ \sum_{\rho \in \wh{G}} d_{\rho} \ang{\wh x(\rho), \wh{f}(\rho) \wh x(\rho) ) }_\HS }
&& \text{\small [Convolution \& Parseval]}
\\
&\le  \sum_{\rho \in \wh{G}} d_{\rho} \snorm{\wh{f}(\rho)}_{S_p} \snorm{\wh x(\rho) }_{S_{2p/(p-1)}}^2 && \text{\small [by \cref{eq:qform-holder-step}]}
\\
&\le \paren{ \sum_\rho d_\rho \snorm{ \wh{f}(\rho) }_{S_p}^p }^{1/p}
    \paren{ \sum_\rho d_\rho \snorm{ \wh x(\rho) }_{S_{2p/(p-1)}}^{2p/(p-1)} }^{(p-1)/p} 
    &&\text{\small  [H\"older's inequality]}
\\
&= \norm{f}_{S_p} \norm{x}_{S_{2p/(p-1)}}^2 &&\text{\small [by \cref{eq:Sp}]}
\\
&\le \norm{f}_{S_p} \norm{x}_{2p/(p+1)}^2. &&\text{\small [by \cref{lem:hy}]} 
\end{align*}
This proves the desired inequality.
\end{proof}

Naor proved the following uniform bound on the Schatten norms of random Cayley graphs via a novel Azuma-type concentration inequality in uniformly smooth normed spaces. 

\begin{lemma}[{\cite[Lemma 4.1]{Nao12}}] \label{lem:schatten-concentration}
There exists a universal constant $C > 0$ with the following property.
For any positive integer $k$ and any finite group $G$, if $g_1, \dots, g_k \in G$ are chosen independently and uniformly at random, then, with probability at least $1/2$, the function $f \colon G \to \RR$ given by
\begin{equation}\label{eq:f-sum}
f = \mbm{1}_{\{g_1\}} + \mbm{1}_{\{g_1^{-1}\}} + \cdots + \mbm{1}_{\{g_k\}} + \mbm{1}_{\{g_k^{-1}\}} - \frac{2k}{\abs{G}}
\end{equation}
satisfies
\[
\norm{f}_{S_p} \le C \abs{G}^{-1+1/p} \sqrt{pk}
\]
simultaneously for every integer $p \ge 2$.
\end{lemma}

Using this concentration lemma, we can now prove \cref{thm:naor-nonabelian}. Unlike the proof in \cite{Nao12} for abelian groups, we do not need to rely on a bounded eigenbasis.

\begin{proof}[Proof of \cref{thm:naor-nonabelian}]
Let $g_1, \dots, g_k \in G$ be the random group elements generating the Cayley (multi)graph. Define $f \colon G \to \RR$ as in \cref{eq:f-sum}.

Let $X \subseteq G$ with $1 < \abs{X} \le \abs{G}/2$. Define a function $x \colon G \to \RR$ by 
\[
x = \abs{G \setminus X} \mbm{1}_X - \abs{X} \mbm{1}_{G \setminus X}.
\]
It is straightforward to check that
\[
\ang{x, f \ast x} = \frac{2k}{\abs{G}} \abs{X}\abs{G\setminus X} - e(X, G \setminus X).
\]
We also have
\[
\norm{x}_{2p/(p+1)} = \paren{\frac{ \abs{X} \abs{G\setminus X}^{\frac{2p}{p+1}} +   \abs{X}^{\frac{2p}{p+1}} \abs{G \setminus X} }{\abs{G}} }^{\frac{p+1}{2p}}.
\]
Applying the inequality $\abs{\ang{ f \ast x, x }} \le \norm{f}_{S_p} \norm{x}_{2p/(p+1)}^2$ from \cref{lem:qform-schatten}, and with the upper bound $\norm{f}_{S_p} \lesssim \abs{G}^{-1+1/p} \sqrt{pk}$ 
from \cref{lem:schatten-concentration}, we obtain that with probability at least $1/2$, one has
\[
\abs{\frac{2k}{\abs{G}} \abs{X}\abs{G\setminus X} - e(X, G \setminus X)} 
\lesssim \abs{G}^{-1+1/p} \sqrt{pk} 
\paren{\frac{ \abs{X} \abs{G\setminus X}^{\frac{2p}{p+1}} +  \abs{X}^{\frac{2p}{p+1}} \abs{G \setminus X}  }{\abs{G}} }^{\frac{p+1}{p}}
\]
simultaneously for all positive integers $p$.
Dividing both sides by $2k \abs{X}\abs{G\setminus X} / \abs{G}$, we obtain
\begin{align*}
\abs{ \frac{e(X, G \setminus X)}{\frac{2k}{\abs{G}} \abs{X}\abs{G\setminus X}}  - 1} 
&
\lesssim \abs{G}^{1/p} \sqrt{\frac{p}{k}} 
\paren{\frac{ \abs{X}^{\frac{1}{p+1}} \abs{G\setminus X}^{\frac{p}{p+1}} +  \abs{X}^{\frac{p}{p+1}}  \abs{G \setminus X}^{\frac{1}{p+1}} }{\abs{G}} }^{\frac{p+1}{p}}
\\
& \lesssim \abs{X}^{1/p} \sqrt{\frac{p}{k}},
\end{align*}
where in the last step we apply the inequality $x^t (1-x)^{1-t} + x^{1-t} (1-x)^t \le x^t + x^{1-t} \le 2x^t$ for $x = \abs{X}/\abs{G} \le 1/2$ and $t = 1/(p+1) \in [0,1/2]$.
Finally, setting $p = \ceil{\log \abs{X}}$, we see the final expression has an upper bound of $O(\sqrt{(\log \abs{X})/k})$.
\end{proof}


\end{document}